\documentclass[9pt,reqno]{amsart}
\usepackage{amsmath, amsthm, amssymb, stmaryrd}
\usepackage{hyperref}
\usepackage{enumerate}
\usepackage{enumitem}
\usepackage{url}
\usepackage{tikz-cd}

\let\OLDthebibliography\thebibliography
\renewcommand\thebibliography[1]{
\OLDthebibliography{#1}
\setlength{\parskip}{0pt}
\setlength{\itemsep}{0pt plus 0.3ex}
}

\usepackage{mathtools}

\usepackage{multirow, array}
\usepackage{placeins}

\usepackage{caption}
\advance \topmargin by -\headheight
\advance \topmargin by -\headsep

\evensidemargin \oddsidemargin
\newtheorem{thm}{Theorem}[section]
\newtheorem{lemma}[thm]{Lemma}
\newtheorem{prop}[thm]{Proposition}

\theoremstyle{definition}
\newtheorem{defn}[thm]{Definition}

\theoremstyle{remark}
\newtheorem{remark}[thm]{Remark}

\numberwithin{equation}{section}

\newcommand*\wrapletters[1]{\wr@pletters#1\@nil}
\def\wr@pletters#1#2\@nil{#1\allowbreak\if&#2&\else\wr@pletters#2\@nil\fi}

\def \bT {\mathbb T}

\def \fm {\mathfrak m}

\def \fM {\mathfrak M}

\def \cP {\mathcal P}

\def \deg {\mathrm{deg}}

\begin{document}
\title[Maximal extension of the Bloom-Maynard bound]{A maximal extension of the Bloom-Maynard bound for sets with no square differences}
\author[Nuno Arala]{Nuno Arala}
\address{Mathematics Institute, Zeeman Building, University of Warwick, Coventry CV4 7AL}
\email{Nuno.Arala-Santos@warwick.ac.uk}
\thanks{}
\date{}
\begin{abstract} We show that if $h\in\mathbb{Z}[x]$ is a polynomial of degree $k$ such that the congruence $h(x)\equiv0\pmod{q}$ has a solution for every positive integer $q$, then any subset of $\{1,2,\ldots,N\}$ with no two distinct elements with difference of the form $h(n)$, with $n$ positive integer, has density at most $(\log N)^{-c\log\log\log N}$, for some constant $c$ that depends only on $k$. This improves on the best bound in the literature, due to Rice, and generalizes a recent result of Bloom and Maynard.
\end{abstract}
\maketitle

\tableofcontents

\section{Introduction}
\label{intro}

The following question was asked by Lovász: if $A\subseteq\mathbb{N}$ is a set such that no two different elements of $A$ differ by a square, must $A$ have upper density $0$? Here, as is customary, we define the upper density of a subset $A\subseteq\mathbb{N}$ as
$$\limsup_{N\to\infty}\frac{|A\cap\{1,\ldots,N\}|}{N}\text{.}$$
This question was settled by Furstenberg \cite{Furstenberg} and Sárközy \cite{Sarkozy}, who proved independently that the answer is affirmative; the former used ergodic theory, while the latter employed a Fourier-analytic approach, based on the Hardy-Littlewood circle method. In recent years there has been interest in strengthening this result by obtaining better bounds for $|A|$, where $A\subseteq\{1,\ldots,N\}$ is a set with no square differences. For example, Sárközy's original argument shows directly that, for such $A$,
\begin{equation}
\label{origbound}
\frac{|A|}{N}\ll\frac{(\log\log N)^{2/3}}{(\log N)^{1/3}}\text{.}
\end{equation}
This has since been much improved. For example, Pintz, Steiger and Szemerédi \cite{PSS} showed that, under the same assumptions,
\begin{equation}
\label{pssbound}\frac{|A|}{N}\ll(\log N)^{-c\log\log\log\log N}
\end{equation}
for some absolute constant $c$. More recently, Bloom and Maynard \cite{BM} devised a powerful bound for the additive energy of a set of rational numbers with small denominators, and applied it to improve the above bound to
\begin{equation}
\label{bmbound}\frac{|A|}{N}\ll(\log N)^{-c\log\log\log N}
\end{equation}
for some absolute constant $c$. This is the best bound up to date.

It is only natural to wonder what happens if one changes the set of forbidden differences, say by replacing it with the set $h(\mathbb{N})$ for some polynomial $h\in\mathbb{Z}[x]$. Lovász's problem concerns the case when $h(x)=x^2$. It is quickly seen that the answer to Lovász's question is no longer affirmative in general. For example, say we take $h(x)=x^2+1$. One may then consider the subset $3\mathbb{N}$ of $\mathbb{N}$, which has upper density $1/3>0$, while on the other hand the difference between any two of its elements is divisible by $3$ and hence not of the form $x^2+1$.

The failure of the above polynomial $h$ to yield an affirmative answer to Lovász's generalized question is, as seen above, due to the existence of an integer $q$ (namely, $3$) such that $q$ does not divide $h(n)$ for any integer $n$. Indeed, if that is the case, then $q\mathbb{N}$ is a subset of $\mathbb{N}$ with positive upper density such that no two of its elements differ by an element of $h(\mathbb{N})$. Polynomials for which no such $q$ exists are called \emph{intersective polynomials} in the literature. We formalize the definition below.

\begin{defn}[Intersective polynomial]
A polynomial $h\in\mathbb{Z}[x]$ is called \emph{intersective} if for any positive integer $q$ the congruence
$$h(n)\equiv0\pmod{q}$$
has an integer solution.
\end{defn}

From the remarks above, it is clear that, in order for there to be any hope of extending the Furstenberg-Sarközy's theorem, or any of its direct descendants, to the polynomial $h$, one must impose that $h$ is intersective. This turns out to be the only restriction needed to ensure an affirmative answer to Lovász's question, as shown by Kamae and Mendès France \cite{KMF}. A result of the quality of \eqref{origbound} has been obtained, for general intersective $h$, by Lucier \cite{Lucier}, with improvements by Lyall and Magyar \cite{LM} and Rice \cite{Rice2}. Balog, Pelikan, Pintz and Szemerédi \cite{BPPE} proved the corresponding version of \eqref{pssbound} for the polynomial $h(x)=x^k$. Recent work of Rice generalizes this, proving a result of the quality of \eqref{pssbound} for a general intersective $h$.

The goal of the present work is to generalize \eqref{bmbound} for arbitrary intersective polynomials. More precisely, we will prove the following:

\begin{thm}
\label{main}
Let $k$ be a positive integer. Then there exists a constant $c=c(k)$ such that, for any intersective polynomial $h\in\mathbb{Z}[x]$ of degree $k$, the following holds: for any positive integer $N$ and any subset $A\subseteq\{1,\ldots,N\}$ such that the equation $a-b=h(n)$ does not have any solutions with $a,b\in A$, $a\neq b$ and $n\in\mathbb{N}$, we have
$$\frac{|A|}{N}\ll_h (\log N)^{-c\log\log\log N}\text{.}$$
\end{thm}

There are several obstacles that one encounters when trying to extend the argument from the $h(n)=n^2$ case to the general case. In the Fourier-analytic argument employed by Bloom and Maynard in \cite{BM}, estimates for quadratic Gauß sums play a role, and it is vital for the argument that square root cancellation is available for such sums. More precisely, if $a$ and $q$ are coprime positive integers, one has
$$\sum_{s=0}^{q-1}e\left(\frac{as^2}{q}\right)\ll q^{1/2}\text{.}$$
However, for a general polynomial $h$ of degree $k$, the estimate
$$\sum_{s=0}^{q-1}e\left(\frac{ah(s)}{q}\right)\ll q^{1-1/k}$$
is best possible. To circumvent this, we adapt an argument of Rice, sieving out admissible residue classes $s\pmod{q}$ such that square root cancellation holds for the modified version of the previous sum where we only sum over these admissible residues.

The other obstacle comes from the density increment strategy that Bloom and Maynard employ, where the equality $(qn)^2=q^2n^2$ plays an important role. To get around this, one needs to adapt the density increment argument in such a way that the intersective polynomial $h$ under consideration changes at every step of the density increment strategy. This will, of course, require that the bounds we obtain exhibit some uniformity with respect to $h$.

\section{Notation and conventions}
Let $\mathbb{T}=\mathbb{R}/\mathbb{Z}$ be the circle group, and let $\|\cdot\|$ denote the usual norm in $\mathbb{T}$. For $1\leq a\leq q$ with $\mathrm{gcd}(a,q)=1$, set
$$\fM_{a,q}\left(N,K\right)=\left\{\gamma\in\bT:\left\|\gamma-\frac{a}{q}\right\|\leq\frac{K}{N}\right\}\text{.}$$
For a choice of parameter $Q$, we then define the \emph{major arcs}
$$\fM(N,K,Q)=\bigcup_{\substack{(a,q)=1\\q\leq Q}}\fM_{a,q}\left(N,K\right)$$
and the \emph{minor arcs}
$$\fm(N,K,Q)=\mathbb{T}\setminus\fM_q(N,K,Q)\text{.}$$

Given a polynomial $h\in\mathbb{Z}[x]$, we say that a set $A$ of integers is \emph{$h$-free} if for any $a\neq b$ in $A$ the difference $a-b$ does not lie in $h(\mathbb{N})$. We use the notation $e(x)$ for $e^{2\pi i x}$. A summation sign with an asterisk $\sum^\ast$ denotes a sum over primitive residue classes.

We use the notation $\widehat{f}$ for the Fourier transform of $f$; so if $f:\mathbb{Z}\to\mathbb{C}$ is a compactly supported function, we define its Fourier transform $\widehat{f}:\mathbb{T}\to\mathbb{C}$ by
$$\widehat{f}(\gamma)=\sum_{n\in\mathbb{Z}}f(n)e(n\gamma)\text{.}$$
The convolution of two (compactly supported, say) functions $f,g:\mathbb{Z}\to\mathbb{C}$ is defined by
$$(f\ast g)(n)=\sum_{a+b=n}f(a)g(b)\text{.}$$
We will use the notation $\mathbf{1}_A$ for the indicator function of a set $A\subseteq\mathbb{Z}$, so that
$$\mathbf{1}_A(n)=\begin{cases}1&\text{ if }n\in A\\0&\text{ otherwise.}\end{cases}$$

In view of the result we want to prove (Theorem \ref{main}), for any fixed constant $c$, given a $h$-free subset $A\subseteq\{1,\ldots,N\}$ with density $\sigma$, one may assume that $\sigma\geq(\log N)^{-c\log\log\log N}$, for otherwise the result follows. This will be assumed implicitly throughout the paper. So, for example, the statement of Lemma \ref{2ndc} takes the form ``Then either $\sigma<(\log N)^{-\varepsilon\log\log\log N}$, or (...)'' but in practice, in applications of Lemma \ref{2ndc}, we will always assume automatically that we are in the second case. We may (and will) also assume that $N$ is larger than any fixed constant.

The letters $c$ and $C$ will be used to denote arbitrary positive constants; we make the usual abuse of notation of allowing them to denote different constants at different times. We hope this will not cause any confusion.

\section{The auxiliary polynomials}
\subsection{Construction of the polynomials}
As was mentioned the introduction, the simplicity of the polynomial $x^2$ makes it particularly suitable for an approach to the problem at hand via the density increment strategy that has become pervasive in Arithmetic Combinatorics over the years. If $A\subseteq[N]$ is a set with no square differences and one manages to locate a progression $q[N']+r$ on which $A$ has increased density, then by further splitting this progression into progressions of modulus $q^2$ one finds another progression, say $q^2[N^\ast]+s$, on which $A$ also has increased density. Defining
$$A^\ast=\{x\in[N^\ast]:q^2x+s\in A\}$$
one sees that $A^\ast$ has no square differences, for if $a\neq b$ are in $A^\ast$ and $a-b=n^2$, then $q^2a+s$ and $q^2b+s$ are in $A$ and
$$(q^2a+s)-(q^2b+s)=q^2(a-b)=(qn)^2\text{,}$$
which contradicts the assumption that $A$ has no square differences. Hence we obtain a new set $A^\ast\subseteq[N^\ast]$ with no square differences and increased density.

For a more general intersective polynomial $h$, given a $h$-free set $A\subseteq[N]$, if we locate a progression on which $A$ of modulus $q$ is dense we cannot expect to obtain from it another set $A'\subseteq[N']$ of increased density with no nonzero differences in $h(\mathbb{N})$. What we can do is to use a construction of Lucier \cite{Lucier}, which allows us to obtain a set $A'\subseteq[N']$ of increased density with no nonzero differences in $h_q(\mathbb{N})$ for some polynomial $h_q\in\mathbb{Z}[x]$. We explain this construction below. We remark that the very existence of these auxiliary polynomials is precisely the step in the argument where the intersectivity of $h$ is used.

Fix once and for all an intersective polynomial $h$ of degree $k$, without loss of generality with positive leading coefficient. The condition that $h$ is intersective is equivalent to the statement that $h$ has a zero $z_p\in\mathbb{Z}_p$ for every prime $p$. Let $m_p$ be the multiplicity of $z_p$ as a zero of $h$. For each positive integer $d$ define $r_d$ to be the only integer satisfying $-d<r_d\leq0$ and
$$r_d\equiv z_p\pmod{p^\alpha}\text{ whenever }p\text{ is a prime and }p^\alpha\mid d\text{.}$$
Clearly $r_d$ exists and is unique by the Chinese Remainder Theorem. We now let $\lambda$ be the unique completely multiplicative function that satisfies $\lambda(p)=p^{m_p}$ for every prime $p$. It is clear from the definition that $d\mid\lambda(d)\mid d^k$ for any $d$. We are now ready to define our auxiliary polynomials.
\begin{defn}[Auxiliary polynomials]
For each $d\in\mathbb{N}$, we define
$$h_d(x)=\frac{h(r_d+dx)}{\lambda(d)}\text{.}$$
\end{defn}
It is an easy matter to check from the definitions that $h_d$ has integer coefficients. It is also immediate from the definition that $\lambda(q)h_q(\mathbb{N})\subseteq h(\mathbb{N})$, which explains why these auxiliary polynomials may act as a replacement for the ``$q^2$ times a square is a square'' property alluded to above. In fact, more is true:
\begin{prop}
\label{rabo}
For any positive integers $d,q$,
$$\lambda(q)h_{dq}(\mathbb{N})\subseteq h_d(\mathbb{N})\text{.}$$
\end{prop}
\begin{proof}
\label{inh}
Since by construction we have $r_{dq}\equiv r_d\pmod{d}$, we can write $r_{dq}=r_d+ds$ for some $-q<s\leq0$. For a positive integer $n$, we then have
$$\lambda(q)h_{dq}(n)=\lambda(q)\frac{h(r_{dq}+dqn)}{\lambda(dq)}=\frac{h(r_{dq}+dqn)}{\lambda(d)}=\frac{h(r_d+ds+dqn)}{\lambda(d)}=h_d(s+qn)\text{.}$$
Since $s+qn>0$ (as $s>-q$), the result follows.
\end{proof}
We register here a straightforward but essential analytic property of the auxiliary polynomials.
\begin{prop}
\label{jb}
For any positive integer $d$, denote by $b_d$ and $J_d$ the leading coefficient of $h_d$ and the sum of the absolute values of the coefficients of $h_d$, respectively. Then $J_d\ll_h b_d$.
\end{prop}
\begin{proof}
Write $h(x)=a_kx^k+\cdots+a_0$. A binomial expansion shows that the the coefficient of $x^i$ in $h_d(x)$ is
\begin{equation}
\label{c}
\frac{1}{\lambda(d)}\sum_{j=i}^ka_j\binom{j}{i}r_d^{j-i}d^i\text{.}
\end{equation}
In particular, since $|r_d|<d$, it follows from the triangle inequality that
$$J_d\ll_h\frac{d^k}{\lambda(d)}\text{.}$$
But it also follows from \eqref{c} that $b_d=a_dd^k/\lambda(d)$, and the result follows.
\end{proof}

\subsection{Sieving good residue classes}
As explained in the introduction, a fundamental obstacle one encounters when trying to apply ideas inspired by the approach of Bloom and Maynard in \cite{BM} is the lack of square root cancellation exhibited by exponential sums $\pmod{q}$ over values of a polynomial of degree at least $3$. Rice \cite{Rice} has found an ingenious way to address this problem. Building on the fact that square root cancellation always holds for squarefree moduli (by the Chinese Remainder Theorem and the Weil bound), he then leverages a classical ``$p$-adic stationary phase'' argument and shows that if we sum only over a certain large set of residue classes at which the derivative of the polynomial has appropriate non-vanishing properties, then we do indeed attain square root cancellation. Below we explain Rice's construction of admissible residue classes, which we will use, and on section \S\ref{expsums} we give the relevant square root cancellation bounds for exponential sums.

\begin{defn}
Let $g\in\mathbb{Z}[x]$ and $Y>0$. For each prime $p$, let $\gamma(g;p)$ be the smallest positive integer for which the derivative $g'$ does not vanish identically modulo $p^{\gamma(g;p)}$.
\begin{enumerate}[label=(\roman*)]
\item We define
$$W(g;Y)=\{n\in\mathbb{Z}:g'(n)\not\equiv0\pmod{p^{\gamma(g;p)}}\text{ for all }p\leq Y\}\text{.}$$
\item Moreover, given also a positive integer $q$, we define
$$W_q(g;Y)=\{n\in\mathbb{Z}:g'(n)\not\equiv0\pmod{p^{\gamma(g;p)}}\text{ for all }p\leq Y\text{ for which }p^{\gamma(g;p)}\mid q\}$$
which intuitively only captures the restrictions that can be read from $n\pmod{q}$.
\item Given a prime $p$, we let $j(g;p)$ be the number of roots of $g'$ modulo $p^{\gamma(p)}$.
\item In case the polynomial $g$ under consideration is one of the $h_d$ we abbreviate $\gamma(g;p)$, $W(g;Y)$, $W_q(g;Y)$ and $j(g;p)$ by $\gamma_d(p)$, $W_d(Y)$, $W^q_d(Y)$, and $j_d(p)$ respectively.
\end{enumerate}
\end{defn}

If $g(x)=a_0+a_1x+\cdots+a_kx^k$, define
$$\mathrm{cont}(g)=\mathrm{gcd}(a_1,\ldots,a_k)\text{.}$$
In \cite{Rice}, section 2.5, Rice proves two key properties of the auxiliary polynomials $h_d$ that will be vital to the application of the exponential sum estimates in the next section. One is the fact that $\mathrm{cont}(h_d)$ is absolutely bounded; hence in practice all the constants that occur in the exponential sum bounds in the next section that depend on $\mathrm{cont}(g)$ are for our purposes absolute constants provided that we only work with the polynomials $h_d$. The other is the fact that the set of moduli $q$ for which $h_d'$ vanishes identically modulo $q$ is contained in a finite set independent of $d$. This implies the following statement.

\begin{prop}
\label{logy}
Let
$$w_d(Y)=\prod_{p\leq Y}\left(1-\frac{j_d(p)}{p^{\gamma_d(p)}}\right)$$
be the expected proportion of integers in $W_d(Y)$. Then we have
$$w_d(Y)\gg_h(\log Y)^{1-k}\text{.}$$
\end{prop}
\begin{proof}
From our previous remark, we deduce that for all but finitely many primes $p$ we have $\gamma_d(p)=1$ for every $d$. Moreover, $h_d'$ has degree $k-1$ and hence has at most $k-1$ roots modulo $p$. It follows that the factor $1-j_d(p)/p^{\gamma_d(p)}$ is at least $1-(k-1)/p$ for all but finitely many primes. Hence, since all factors are clearly nonzero,
$$w_d(Y)\gg_h\prod_{k\leq p\leq Y}\left(1-\frac{k-1}{p}\right)\gg(\log Y)^{1-k}\text{,}$$
by a classical result of Mertens.
\end{proof}

In fact Rice makes precise in \cite{Rice}, via a standard Brun sieve argument, the intuition that $w_d(Y)$ should be close to the proportion of integers in $[1,X]$ lying in $W_d(Y)$, provided that $X$ is sufficiently large in terms of $Y$. More precisely, he proves the following:

\begin{lemma}
\label{sievelemma}
Let $g\in\mathbb{Z}[x]$ and $X,Y>0$ with $c\log X\geq\log Y\log\log Y$. Then,
$$|[1,X]\cap W(g;Y)|=X\prod_{p\leq Y}\left(1-\frac{j(g;p)}{p^{\gamma(g;p)}}\right)+O\left(Xe^{-c\frac{\log X}{\log Y}}\right)\text{,}$$
where $c>0$ depends only on $\mathrm{deg}(g)$ and the collection of moduli for which $g'$ vanishes identically.
\end{lemma}
\begin{proof}
This is \cite{Rice}, Proposition 2.4.
\end{proof}

\section{Exponential sum estimates}
\label{expsums}
In this section we record some estimates for exponential sums arising from a polynomial $g$ of degree $k$, with the main goal being to establish that certain exponential sums are small on appropriate minor arcs. We start by recalling some results of Rice \cite{Rice} on which we will build.

\begin{lemma}
\label{expsumi}
Suppose $k\in\mathbb{N}$, $g(x)=a_0+a_1x+\cdots+a_kx^k\in\mathbb{Z}[x]$, and let $J=|a_0|+\cdots+|a_k|$. If $X,Y>0$, $a,q\in\mathbb{N}$, $\alpha=a/q+\beta$, and $c\log(X/q)\geq\log Y\log\log Y$, then
\begin{align*}
\sum_{\substack{n=1\\n\in W(g;Y)}}^Xg'(n)e(g(n)\alpha)&=\frac{1}{q}\prod_{\substack{p\leq Y\\p^{\gamma(g;p)}\nmid q}}\left(1-\frac{j(g;p)}{p^{\gamma(g;p)}}\right)\sum_{\substack{s=0\\s\in W^q(g;Y)}}^{q-1}e\left(\frac{ag(s)}{q}\right)\int_0^X g'(x)e(g(x)\beta)dx\\
&+O\left(kJX^ke^{-c\frac{\log\left(\frac{X}{q}\right)}{\log Y}}(1+JX^k|\beta|)\right)\text{,}\\
\end{align*}
where $c=c(k,\mathrm{cont}(g))>0$.
\end{lemma}

\begin{proof}
This is Lemma 4.1 in \cite{Rice}.
\end{proof}

\begin{lemma}
\label{expsumii}
If $g\in\mathbb{Z}[x]$ with $\deg(g)=k\geq 2$, $a,q\in\mathbb{N}$ with $\mathrm{gcd}(a,q)=1$, and $Y>0$, then
$$\left|\sum_{\substack{s=0\\s\in W^q(g;Y)}}^{q-1}e\left(\frac{ag(s)}{q}\right)\right|\ll_k\mathrm{gcd}(\mathrm{cont}(g),q)^3C^{\omega(q)}\begin{cases}q^{1/2}\text{ if }q\leq Y\\q^{1-1/k}\text{ for all }q\end{cases}\text{,}$$
where $C=C(k)$ and $\omega(q)$ is the number of distinct prime factors of $q$.
\end{lemma}

\begin{proof}
This is Lemma 4.3 in \cite{Rice}.
\end{proof}

\begin{lemma}
\label{expsumiii}
Suppose $k\in\mathbb{N}$, $g(x)=a_0+a_1x+\cdots+a_kx^k\in\mathbb{Z}[x]$ with $a_k>0$. Suppose further that $X,Y,Z\geq2$, $YZ\leq X$, and $a,q\in\mathbb{N}$ with $\mathrm{gcd}(a,q)=1$. If $|\alpha-a/q|<1/q^2$, then
$$\left|\sum_{\substack{n=1\\n\in W(g;Y)}}^Xe(g(n)\alpha)\right|\ll_k\mathrm{cont}(g)^6(\log Y)^{ek}X\left(e^{-\frac{\log Z}{\log Y}}+\left(a_k\log^{k^2}(a_kqX)\left(q^{-1}+\frac{Z}{X}+\frac{qZ^k}{a_kX^k}\right)\right)^{2^{-k}}\right)\text{.}$$
\end{lemma}

\begin{proof}
This is Lemma 4.5 in \cite{Rice}.
\end{proof}

We will use these to prove the following key estimate. In the following, set $\rho=1/2^{k+2}k^2$.

\begin{lemma}
\label{1stc}
Suppose $N\gg_{h}1$, and let $r>0$ be arbitrary. Let $d$ be such that $d\ll N^\rho$. Set $M=\left\lfloor(N/b_d)^{1/k}\right\rfloor$. Then either $\sigma\leq(\log N)^{-\log\log\log N/4}$, or there exists a constant $\kappa$ (possibly depending on $r$) such that, setting $Y=\sigma^{-(k+2)}$,
$$\left|\frac{1}{w_d(Y)N}\sum_{\substack{m=1\\n\in W_d(Y)}}^{M}h_d'(m)e\left(h_d(m)\gamma\right)\right|\leq r\sigma$$ whenever $\gamma\in\fm(N,\kappa/\sigma,\kappa/\sigma^{k+1})$.
\end{lemma}

\begin{remark}
The denominator $w_d(Y)N$ is the trivial bound for the exponential sum, so that the Lemma says something about how much we can improve over the trivial bound.
\end{remark}

\begin{proof}
Let $S$ be the expression we wish to bound. Throughout the argument, we abbreviate $w_d(Y)$ by $w$, and let $Z=e^{(\log\log N)^3}$. Dirichlet's approximation theorem yields a positive integer $q$, with $1\leq q\leq N/Z^{2k}$, and an integer $a$ with $1\leq a\leq q$ coprime to $q$, such that
$$\left|\gamma-\frac{a}{q}\right|<\frac{Z^{2k}}{qN}\text{.}$$
Write $\gamma=a/q+\beta$. We split the argument into several cases.

Suppose first that $q\leq N^{1/2k}$. Recalling that $J_d\ll b_d$ (Proposition \ref{jb}) and that hence $J_d M^d\ll N$, Lemma \ref{expsumi} together with Proposition \ref{logy} yields
$$S=\frac{w_q}{qN}\sum_{\substack{s=0\\s\in W_d^q(Y)}}^{q-1}e\left(\frac{ah_d(s)}{q}\right)\int_0^Mh_d'(x)e(h_d(x)\beta)dx+O_h\left((\log Y)^{k-1}e^{-c\frac{\log(M/q)}{\log Y}}Z^{2k}\right)\text{,}$$
where
$$w_q=\prod_{\substack{p\leq Y\\p^{\gamma_d(p)\mid q}}}\left(1-\frac{j_d(p)}{p^{\gamma_d(p)}}\right)^{-1}\text{.}$$
Since $q\leq N^{1/2k}$ we have $\log(M/q)\asymp\log(N)$, and hence
$$(\log Y)^{k-1}e^{-c\frac{\log(M/q)}{\log Y}}Z^{2k}\ll(\log(\sigma^{-1}))^{k-1}e^{-c\frac{\log(N)}{\log(\sigma^{-1})}}e^{2k(\log\log N)^3}$$
and if $\sigma^{-1}\leq(\log N)^{\log\log\log N/4}$, then the above is $O(\sigma^t)$ for any $t>0$, and in particularly is $o(\sigma)$. Hence we can focus on controlling the main term. For that, we observe that
$$\left|\int_0^Mh_d'(x)e(h_d(x)\beta)dx\right|=\left|\int_{h_d(0)}^{h_d(M)}e(y\beta)dy\right|\ll\min\{N,|\beta|^{-1}\}\text{,}$$
provided that $h_d(0)\ll N$; but since by definition $h_d(0)=h(r_d)/\lambda(d)\ll d^{k-1}$ (recall that $|r_d|<d$), this condition certainly holds, taking into account that we are assuming $d\ll N^{\rho}$. Hence we have
\begin{equation}
\label{rec}
S\ll\frac{w_q}{qN}\sum_{\substack{s=0\\s\in W_d^q(Y)}}^{q-1}e\left(\frac{ah_d(s)}{q}\right)\min\{N,|\beta|^{-1}\}+o_{h,N}(\sigma)\text{.}
\end{equation}
We now split into two further sub-cases.

If $q>\kappa/\sigma^{k+1}$, then Lemma \ref{expsumii}, together with the observation that $w_q\ll_h 2^{\omega(q)}$, guarantees that the exponential sum is bounded above by $C^{\omega(q)}q^{1-1/k}$ for some constant $C$, which implies that
\begin{align*}S&=O\left(\frac{C^{\omega(q)}}{q^{1/k}}\right)+o_{h,N}(\sigma)\\
&=O\left(\frac{1}{q^{1/(k+1)}}\right)+o_{h,N}(\sigma)\\
\end{align*}
by using the trivial bound $C^{\omega(q)}\ll q^{1/k-1/(k+1)}$. Using $q>\kappa/\sigma^{k+1}$ yields now
$$S=O\left(\sigma/\kappa^{1/(k+1)}\right)+o_{h,N}(\sigma)$$
which gives the desired result in this case for an appropriate choice of $\kappa$.

If instead $q\leq\kappa/\sigma^{k+1}$, then the assumption that $\gamma\in\fm(N,\kappa/\sigma,\kappa/\sigma^{k+1})$ guarantees that $|\beta|\geq\frac{\kappa}{\sigma N}$. Bounding the exponential sum trivially this time and bounding $\min\{N,|\beta|^{-1}\}$ from above by $|\beta|^{-1}$, we obtain
$$S=O\left(\frac{|\beta|^{-1}}{N}\right)+o_{h,N}(\sigma)=O\left(\frac{\sigma}{\kappa}\right)+o_{h,N}(\sigma)\text{.}$$
This gives the desired result in this case, by choosing an appropriate value of $\kappa$.

We now tackle the case when $q>N^{1/2k}$. Using partial summation, we may write
$$S=\frac{1}{wN}\left(h_d'(M)\sum_{\substack{n=1\\n\in W_d(Y)}}^Me(h_d(n)\gamma)-\int_1^Mh_d''(t)\left(\sum_{\substack{1\leq n\leq t\\n\in W_d(Y)}}e(h_d(n)\gamma)\right)dt\right)\text{.}$$
The assumptions
$$\frac{N}{Z^{2k}}>q>N^{1/2k},\quad d<N^\rho,\quad\sigma>(\log N)^{-\log\log\log N/4}\text{,}$$
together with the fact that $b_d\ll d^k$ are easily seen to imply that
\begin{equation}
\label{truck}(\log Y)^{ek}X\left(e^{-\frac{\log Z}{\log Y}}+\left(b_d\log^{k^2}(b_dqX)\left(q^{-1}+\frac{Z}{X}+\frac{qZ^k}{b_dX^k}\right)\right)^{2^{-k}}\right)\ll wM\sigma
\end{equation}
for any $X$ with $\sigma M<X\leq M$. This is the step in the argument that justifies the choice of the exponent $\rho$; it is chosen so that $b_dq^{-2^{-k}}$ is bounded above by a negative power of $N$, namely $N^{k\rho-1/(k\cdot 2^{k+1})}$, and hence yields a power saving against which logarithmic factors and powers of $\sigma^{-1}$ (which are smaller than any positive power of $N$) cannot compete. It also justifies the exponent of $3$ (in fact any exponent bigger than $2$ would suffice) in our choice of $Z$; it makes $e^{\log Z/\log Y}$ bigger than any power of $\sigma^{-1}$. Observe also that $YZ\leq\sigma M$ if $N$ is large enough. Applying this with $X=M$, together with Lemma \ref{expsumiii} and observing that $h_d'(t)\ll b_dt^{k-1}$, yields immediately
$$h_d'(M)\sum_{\substack{n=1\\n\in W_d(Y)}}^Me(h_d(n)\gamma)\ll b_dM^dw\sigma\asymp wN\sigma\text{.}$$
We now decompose
\begin{align*}
&\int_1^Mh_d''(t)\left(\sum_{\substack{1\leq n\leq t\\n\in W_d(Y)}}e(h_d(n)\gamma)\right)dt\\
&=\int_1^{\sigma M}h_d''(t)\left(\sum_{\substack{1\leq n\leq t\\n\in W_d(Y)}}e(h_d(n)\gamma)\right)dt+\int_{\sigma M}^Mh_d''(t)\left(\sum_{\substack{1\leq n\leq t\\n\in W_d(Y)}}e(h_d(n)\gamma)\right)dt\text{.}\\
\end{align*}
For the first integral, we observe that $h_d''(t)\ll b_dM^{k-2}$ for $1\leq t\leq\sigma M$ and bound the exponential sum trivially by $\sigma M$, yielding that
$$\int_1^{\sigma M}h_d''(t)\left(\sum_{\substack{1\leq n\leq t\\n\in W_d(Y)}}e(h_d(n)\gamma)\right)dt\ll b_dM^d\sigma^2\ll wN\sigma\text{.}$$
Finally, we use \eqref{truck}, Lemma \ref{expsumiii} and the bound $h_d''(t)\ll b_dM^{k-2}$ to conclude that
$$\int_{\sigma M}^Mh_d''(t)\left(\sum_{\substack{1\leq n\leq t\\n\in W_d(Y)}}e(h_d(n)\gamma)\right)dt\ll M\cdot b_dM^{k-2}\cdot wM\sigma\asymp wN\sigma\text{.}$$
Together, these imply that $S=O(\sigma)$ in this case, finishing the proof.
\end{proof}

\begin{lemma}
\label{2ndc}
There exists $\varepsilon=\varepsilon(k)>0$ for which the following statement is true. Keep the notation and assumptions of Lemma \ref{1stc}. Fix $q\leq\kappa/\sigma^{k+1}$ and some $a$ with $1\leq a\leq q$ and $\mathrm{gcd}(a,q)=1$. Then either $\sigma<(\log N)^{-\varepsilon\log\log\log N}$, or the following hold:
\begin{enumerate}[label=(\roman*)]
\item For any $\gamma\in\fM_{a,q}(N,\kappa/\sigma)$,
$$\left|\frac{1}{wN}\sum_{\substack{m=1\\n\in W_d(Y)}}^{M}h_d'(m)e\left(h_d(m)\gamma\right)\right|\ll(\log N)^{1/4}q^{-1/2}\min\left\{1,(N|\beta|)^{-1}\right\}\text{.}$$
\item We have
$$\int_{\fM_{a,q}(N,\kappa/\sigma)}\left|\frac{1}{wN}\sum_{\substack{m=1\\n\in W_d(Y)}}^{M}h_d'(m)e\left(h_d(m)\gamma\right)\right|^2d\gamma\ll \frac{(\log N)^{1/2}}{qN}\text{.}$$
\end{enumerate}
\end{lemma}
\begin{proof}
We recover the first part of the proof of Lemma \ref{1stc}, namely \eqref{rec}. Since $q\leq Y$, we can bound the exponential sum by $C^{\omega(q)}q^{1/2}$, according to Lemma \ref{expsumii}. Using the bounds $w_q\ll 2^{\omega(q)}$ and $\omega(q)\ll\log q/\log\log q$, we obtain
$$S=O\left(\frac{C^{\log q/\log \log q}}{q^{1/2}}\min\{1,(N|\beta|)^{-1}\}\right)+O_t(\sigma^t)$$
(note that we had remarked earlier that the error term $o(\sigma)$ was actually $O(\sigma^t)$ for any $t>0$). The error term can hence be absorbed into the main term (recall that both $q$ and $\beta$ are bounded by powers of $\sigma^{-1}$) and hence
$$S=O\left(\frac{C^{\log q/\log \log q}}{q^{1/2}}\min\{1,(N|\beta|)^{-1}\}\right)\text{.}$$
We now choose $\varepsilon>0$ so that the assumption that $\sigma\geq(\log N)^{-\varepsilon\log\log\log N}$ implies
$$C^{\log q/\log\log q}\ll(\log N)^{1/4}\quad\text{for any }q\leq\kappa/\sigma^{k+1}\text{.}$$
It can easily be checked that such an $\varepsilon$ can be found; indeed, since $\log q/\log\log q$ is increasing, the assumption that $q\ll\sigma^{-(k+1)}$ implies that
\begin{align*}
\frac{\log q}{\log\log q}&\ll\frac{(k+1)\log(\sigma^{-1})}{\log\log(\sigma^{-1})}+O(1)\\
&\leq\frac{(k+1)\varepsilon\log\log N\cdot\log\log\log N}{\log(\varepsilon)+\log\log\log N+\log\log\log\log N}+O(1)\leq\varepsilon\log\log N+O(1)
\end{align*}
if $N$ is large enough, which in turn yields
$$C^{\log q/\log\log q}\ll C^{(k+1)\varepsilon\log\log N}=(\log N)^{(k+1)\varepsilon\log(C)}$$
and so it suffices to choose $\varepsilon$ so that $(k+1)\varepsilon\log(C)\leq1/4$. This is, as a matter of fact, perhaps the part where the argument is the closest to breaking down; in order for this strategy to work we need to be able to obtain a small power of $\log N$ in our estimates in this Lemma, and assuming the negation of what we want to prove is ``just enough'' to obtain precisely that.

Then either $\sigma<(\log N)^{-\varepsilon\log\log\log N}$, or
$$S\ll(\log N)^{1/4}q^{-1/2}\min\{1,(N|\beta|)^{-1}\}\text{,}$$
thereby proving (i).
Now (ii) follows directly by applying (i) pointwise. Indeed, we have
\begin{align*}
\int_{\fM_{a,q}(N,\kappa/\sigma)}\left|\frac{1}{wN}\sum_{\substack{m=1\\n\in W_d(Y)}}^{M}h_d'(m)e\left(h_d(m)\gamma\right)\right|^2d\gamma&\ll(\log N)^{1/2}q^{-1}\int_0^\infty(\min\{1,(N|\beta|)^{-1}\})^2d\beta\\
&=(\log N)^{1/2}q^{-1}\left(\int_0^{1/N}1d\beta+\int_{1/N}^\infty(N|\beta|)^{-2}d\beta\right)\\
&=\frac{2(\log N)^{1/2}}{qN}\text{,}
\end{align*}
which implies what we wanted.
\end{proof}

\section{The density increment}
The purpose of this section is twofold. First we prove a variation (Lemma \ref{di}) of the $L^2$ density increment lemma used in \cite{BM}, adapted to our setup. Secondly we prove (Lemma \ref{rdi}) that if $A\subseteq[N]$ is a $h_d$-free set of density $\sigma$ then $\widehat{\mathbf{1}_A}-\sigma\widehat{\mathbf{1}_N}$ has big $L^2$ mass around certain rationals with small denominator, at which the Fourier coefficients of $A$ are large. We then put the two together in Lemma \ref{cor0}, where we establish that, in the absence of a good density increment, each denominator appears a relatively small number of times among the rational numbers in the conclusion of Lemma \ref{rdi}. This will pave the way to an application of the additive energy bound of Bloom and Maynard (Theorem 2 in \cite{BM}).

\begin{lemma}
\label{di}
Let $A\subseteq[N]$ have density $\sigma=|A|/N$, and suppose that $A$ is $h_d$-free. Let $\nu\in[0,1)$ and let $K,q\geq1$ be such that $\nu\sigma N\gg K\lambda(q)$, and suppose that
$$\sum_{\substack{1\leq a\leq q\\(a,q)=1}}\int_{\fM_{a,q}(N,K)}\left|\widehat{\mathbf{1}_{A}}(\gamma)-\sigma\widehat{\mathbf{1}_{[N]}}(\gamma)\right|^2d\gamma\geq\nu\sigma|A|\text{.}$$
Then there exists $N^\ast\gg\nu\sigma N/K\lambda(q)$ and a $h_{qd}$-free subset $A^\ast\subseteq[N^\ast]$, the density $\sigma^\ast=|A^\ast|/N^\ast$ of which satisfies
$$\sigma^\ast \geq\left(1+\frac{\nu}{73}\right)\sigma\text{.}$$
\end{lemma}

\begin{proof}
Let
$$\cP=\{q,2q,\ldots,N'q\}$$
where $N'$ will be chosen later according to some constraints that we will find along the argument. Suppose $\gamma\in\fM_{a,q}(N,K)$, and observe that, for any $n'\in[N']$,
$$\left|1-e(\gamma qn')\right|\ll\|\gamma qn'\|=\|\gamma qn'-an'\|=\left\|q\left(\gamma-\frac{a}{q}\right)n'\right\|\leq\frac{qKN'}{N}\text{.}$$
The above implies that $\left|1-e(\gamma qn')\right|<1/2$ if $N'\ll N/qK$, for some sufficiently small value of the implied constant. Assuming this, we obtain that, for $\gamma\in\fM_{a,q}(N,K)$,
$$\left|N'-\widehat{\mathbf{1}_{\cP}}(\gamma)\right|=\left|\sum_{n\in\cP}(1-e(\gamma n))\right|\leq\frac{|\cP|}{2}=\frac{N'}{2}\text{,}\text{ and therefore }\left|\widehat{\mathbf{1}_{\cP}}(\gamma)\right|\geq\frac{N'}{2}\text{.}$$
Let $g=\mathbf{1}_A-\sigma\mathbf{1}_{[N]}$. By Plancherel's Theorem and the fact that $\widehat{\mathbf{1}_\cP*g}=\widehat{\mathbf{1}_\cP}\widehat{g}$,
\begin{align*}\|\mathbf{1}_\cP*g\|_2^2&=\int_\mathbb{T}\left|\widehat{\mathbf{1}_\cP}(\gamma)\right|^2\left|\widehat{g}(\gamma)\right|^2d\gamma\\
&\geq\int_{\fM_{q}(N,K)}\left|\widehat{\mathbf{1}_\cP}(\gamma)\right|^2\left|\widehat{g}(\gamma)\right|^2d\gamma\\
&\geq\frac{N'^2}{4}\int_{\fM_q(N,K)}\left|\widehat{g}(\gamma)\right|^2d\gamma\geq\frac{\nu\sigma N'^2|A|}{4}\text{.}
\end{align*}
Writing $\mathbf{1}_\cP*g=\mathbf{1}_\cP*\mathbf{1}_A-\sigma\mathbf{1}_\cP*\mathbf{1}_{[N]}$ and expanding the inner product, we arrive at
\begin{equation}
\label{expand}
\left\|\mathbf{1}_\cP*\mathbf{1}_A\right\|_2^2\geq\frac{\nu\sigma N'^2|A|}{4}+2\sigma\langle\mathbf{1}_\cP*\mathbf{1}_A,\mathbf{1}_\cP*\mathbf{1}_{[N]}\rangle-\sigma^2\left\|\mathbf{1}_\cP*\mathbf{1}_{[N]}\right\|_2^2\text{.}
\end{equation}
We now observe that $\langle\mathbf{1}_\cP*\mathbf{1}_A,\mathbf{1}_\cP*\mathbf{1}_{[N]}\rangle$ counts solutions $(p,a,p',n)\in\cP\times A\times\cP\times [N]$ to the equation $p+a=p'+n$. Rewriting the equation as $a=p'-p+n$, we see that
$$\langle\mathbf{1}_\cP*\mathbf{1}_A,\mathbf{1}_\cP*\mathbf{1}_{[N]}\rangle=\sum_{x\in A}\mathbf{1}_\cP*\mathbf{1}_{-\cP}*\mathbf{1}_{[N]}(x)\text{.}$$
It follows that
\begin{align*}
\left|\langle\mathbf{1}_\cP*\mathbf{1}_A,\mathbf{1}_\cP*\mathbf{1}_{[N]}\rangle-|A|N'^2\right|&=\left|\sum_{x\in A}(\mathbf{1}_\cP*\mathbf{1}_{-\cP}*\mathbf{1}_{[N]}(x)-N'^2\mathbf{1}_{[N]}(x))\right|\\
&\leq\sum_{x\in\mathbb{Z}}\left|\mathbf{1}_\cP*\mathbf{1}_{-\cP}*\mathbf{1}_{[N]}(x)-N'^2\mathbf{1}_{[N]}(x)\right|\\
&\leq\sum_{x\in\mathbb{Z}}\left|\sum_{y\in\mathbb{Z}}(\mathbf{1}_\cP*\mathbf{1}_{-\cP})(y)\mathbf{1}_{[N]}(x-y)-\sum_{y\in\mathbb{Z}}(\mathbf{1}_\cP*\mathbf{1}_{-\cP})(y)\mathbf{1}_{[N]}(x)\right|\\
&\leq\sum_{y\in\mathbb{Z}}(\mathbf{1}_\cP*\mathbf{1}_{-\cP})(y)\sum_{x\in\mathbb{Z}}\left|\mathbf{1}_{[N]}(x-y)-\mathbf{1}_{[N]}(x)\right|\\
&\ll\sum_{y\in\mathbb{Z}}(\mathbf{1}_\cP*\mathbf{1}_{-\cP})(y)|y|\\
&\ll qN'\sum_{y\in\mathbb{Z}}(\mathbf{1}_\cP*\mathbf{1}_{-\cP})(y)=qN'^3\text{.}
\end{align*}
It follows that $\langle\mathbf{1}_\cP*\mathbf{1}_A,\mathbf{1}_\cP*\mathbf{1}_{[N]}\rangle=|A|N'^2+O(qN'^3)$. On the other hand, since $\mathbf{1}_\cP*\mathbf{1}_{[N]}(x)$ counts solutions $(p,n)\in\cP\times[N]$ to $p+n=x$, and is hence bounded by $|\cP|=N'$, we have
$$\left\|\mathbf{1}_\cP*\mathbf{1}_{[N]}\right\|_2^2\leq N'\left\|\mathbf{1}_\cP*\mathbf{1}_{[N]}\right\|_1=N'^2N=\frac{|A|N'^2}{\sigma}\text{.}$$
Using these estimates in \eqref{expand}, we arrive at
$$\left\|\mathbf{1}_\cP*\mathbf{1}_A\right\|_2^2\geq \sigma|A|N'^2+\frac{\nu\sigma N'^2|A|}{4}+O(\sigma qN'^3)\text{.}$$
If $N'\ll \nu\sigma N/q$ for some sufficiently small value of the implied constant, it follows that the term $O(\sigma q N'^3)$ is at most $\nu\sigma N'^2|A|/8$, and therefore
$$\left\|\mathbf{1}_\cP*\mathbf{1}_A\right\|_2^2\geq\left(1+\frac{\nu}{73}\right)\sigma N'^2|A|\text{.}$$
We now observe that $\left\|\mathbf{1}_\cP*\mathbf{1}_A\right\|_2^2$ counts solutions $(p,a,p',a')\in\cP\times A\times\cP\times A$ to $p+a=p'+a'$. Rewriting the equation as $p-a'=p'-a$, we see that $\left\|\mathbf{1}_\cP*\mathbf{1}_A\right\|_2^2=\left\|\mathbf{1}_\cP*\mathbf{1}_{-A}\right\|_2^2$. We therefore obtain
$$\left\|\mathbf{1}_\cP*\mathbf{1}_{-A}\right\|_1\cdot\left\|\mathbf{1}_\cP*\mathbf{1}_{-A}\right\|_\infty\geq \left\|\mathbf{1}_\cP*\mathbf{1}_{-A}\right\|_2^2\geq\left(1+\frac{\nu}{73}\right)\sigma N'^2|A|\text{.}$$
On the other hand, it is clear that $\left\|\mathbf{1}_\cP*\mathbf{1}_{-A}\right\|_1=|\cP|\times|A|=N'|A|$, and therefore we conclude that
$$\left\|\mathbf{1}_\cP*\mathbf{1}_{-A}\right\|_\infty\geq\left(1+\frac{\nu}{73}\right)\sigma N'\text{.}$$
Letting $x$ be an integer at which $\mathbf{1}_\cP*\mathbf{1}_{-A}$ attains its maximum, this means precisely that $|\cP\cap (A+x)|\geq(1+\nu/73)\sigma N'$.

We recall that the constraints on $N'$ we found along the way, in order to obtain the conclusion above, were that $N'\ll N/qK$ and $N'\ll\nu\sigma N/q$. It is therefore clear that we can find $N'\gg \nu\sigma N/qK$ for which the conclusion holds. Now partition $[N']$ into residue classes modulo $\lambda(q)/q$. Each class contains $\gg \nu\sigma N/\lambda(q)K$ elements, and by the Pigeonhole Principle there is a class $C$, with $|C|=N^\ast$, such that
$$|qC\cap (A+x)|\geq\left(1+\frac{\nu}{73}\right)\sigma N^\ast\text{.}$$
Let $qC=\{\lambda(q)n+r:1\leq n\leq N^\ast\}$ for some $r\in\mathbb{Z}$, and let
$$A^\ast=\{n\in[N^\ast]:\lambda(q)n+r\in A+x\}\text{.}$$
Note that $|A^\ast|=|qC\cap (A+x)|$. So it remains only to prove that $A^\ast$ is $h_{dq}$-free. Assume it is not, so that there exist $a\neq b$ in $A^\ast$ and a positive integer $n$ such that $a-b=h_{dq}(n)$. Then $\lambda(q)a+r-x$ and $\lambda(q)b+r-x$ are in $A$, and
$$(\lambda(q)a+r-x)-(\lambda(q)b+r-x)=\lambda(q)(a-b)=\lambda(q)h_{dq}(n)\in h_d(\mathbb{N})\text{,}$$
by Proposition \ref{rabo}, contradicting the assumption that $A$ is $h_d$-free. This finishes the proof.
\end{proof}

\begin{lemma}
\label{rdi}
Let $N\gg_h 1$, and suppose $A\subseteq[N]$ is $h_d$-free, where $d\leq N^{\rho}$, with $\sigma=|A|/N>(\log N)^{-\log\log\log N}$. Then there exist $B, Q$, with $\sigma^{O(1)}\ll B\ll\sigma^{-O(1)}$ and $1\leq Q\ll\sigma^{-O(1)}$, and a set
$$\Gamma\subseteq\{(a,q):q\leq Q,1\leq a\leq q,\mathrm{gcd}(a,q)=1\}$$
such that:
\begin{enumerate}[label=(\roman*)]
\item There is a choice of frequencies $\gamma_{a,q}\in\fM_{a,q}(N,\kappa/\sigma)$ for each $(a,q)\in\Gamma$ for which
$$\sum_{(a,q)\in\Gamma}|\widehat{\mathbf{1}_A}(\gamma_{a,q})|\gg\frac{B|A|Q^{1/2}}{(\log N)^{1/4}(\log(1/\sigma))^2}\text{;}$$
\item Setting $g=\mathbf{1}_A-\sigma\mathbf{1}_{[N]}$, we have, for each $(a,q)\in\Gamma$,
$$\int_{\fM_{a,q}(N,\kappa/\sigma)}|\widehat{g}(\gamma)|^2d\gamma\gg\frac{\sigma|A|}{B^2}\text{.}$$
\end{enumerate}
\end{lemma}
\begin{proof}
Let $M=\left\lfloor (N/b_d)^{1/k}\right\rfloor$. For $\gamma\in\mathbb{T}$, define
$$S(\gamma)=\frac{1}{wN}\sum_{\substack{m=1\\m\in W_d(Y)}}^Mh_d'(m)e(h_d(m)\gamma)\text{.}$$
We then observe that, by orthogonality,
$$\int_\bT\widehat{\textbf{1}_A}(\gamma)\overline{\widehat{\textbf{1}_A}(\gamma)}S(\gamma)d\gamma=\frac{1}{wN}\sum_{\substack{1\leq m\leq M\\m\in W_d(Y)\\a,b\in A\\b-a=h_d(m)}}h_d'(m)\text{.}$$
The assumption that $A$ is $h_d$-free implies that only the terms with $h_d(m)=0$ (and hence $a=b$) contribute to the sum above. Note that $h_d(m)=0$ implies, by definition of $h_d$, that $h(r_d+dm)=0$; since $h$ has only finitely many zeros we must have $r_d+dm\ll_h 1$, and since $|r_d|<d$ by assumption, it follows that $m\ll_h 1$. Since $b_d\leq d^k<N^{k\rho}$, it follows that each term $h_d'(m)$ in the sum above contributes at most $O_h(N^{k\rho})$, and hence
\begin{equation}
\label{ort}
\int_\bT\widehat{\textbf{1}_A}(\gamma)\overline{\widehat{\textbf{1}_A}(\gamma)}S(\gamma)d\gamma=\frac{1}{wN}\sum_{\substack{1\leq m\leq M\\m\in W_d(Y)\\a,b\in A\\b-a=h_d(m)}}h_d'(m)=O_h\left(\frac{1}{wN^{1-k\rho}}\right)\text{.}
\end{equation}

Recall that $g=\mathbf{1}_A-\sigma\mathbf{1}_{[N]}$. We claim that
\begin{equation}
\label{part}\int_{\bT}|\widehat{g}(\gamma)\widehat{\mathbf{1}_A}(\gamma)S(\gamma)|d\gamma\gg\sigma|A|\text{.}
\end{equation}
To see this, we split into two cases, according to whether $|A\cap[N/2,N]|>|A|/2$ or $|A\cap[1,N/2]|>|A|/2$. Assuming the former, we use \eqref{ort} to derive that
\begin{align*}
-\int_{\bT}\widehat{g}(\gamma)\overline{\widehat{\mathbf{1}_A}(\gamma)}S(\gamma)d\gamma&=\sigma\int_\bT\mathbf{1}_{[N]}(\gamma)\overline{\widehat{\mathbf{1}_A}(\gamma)}S(\gamma)d\gamma-O\left(\frac{1}{wN^{1-k\rho}}\right)\\
&=\frac{\sigma}{wN}\sum_{\substack{m\in W_d(Y)\cap[1,M]\\a\in[N],b\in A\\a-b+h_d(m)=0}}h_d'(m)-O\left(\frac{1}{wN^{1-k\rho}}\right)\text{.}\\
\end{align*}
We now estimate the contribution to the above sum of those values of $m$ for which $h_d'(m)<0$. Recalling the definition of $h_d$, we see that
$$h_d'(m)=\frac{d}{\lambda(d)}h'(r_d+dm)\text{.}$$
Since $\lambda(d)\geq d$, $|r_d|<d$ and the leading coefficient of $h_d$ is assumed to be positive, the condition $h_d'(m)<0$ implies that $m\ll_h 1$ and also that $|h_d'(m)|\ll_h 1$. Therefore the values of $m$ for which $h_d'(m)<0$ subtract at most $O_h(|A|)=O_h(\sigma N)$ from the sum above, and hence
\begin{align*}-\int_{\bT}\widehat{g}(\gamma)\overline{\widehat{\mathbf{1}_A}(\gamma)}S(\gamma)d\gamma&=\frac{\sigma}{wN}\sum_{\substack{m\in W_d(Y)\cap[1,M]\\a\in[N],b\in A\\a-b+h_d(m)=0\\h_d'(m)>0}}h_d'(m)-O\left(\frac{\sigma^2}{w}\right)-O\left(\frac{1}{wN^{1-k\rho}}\right)\text{.}\\
&=\frac{\sigma}{wN}\sum_{\substack{m\in W_d(Y)\cap[1,M]\\a\in[N],b\in A\\a-b+h_d(m)=0\\h_d'(m)\geq0}}h_d'(m)-O\left(\frac{\sigma^2}{w}\right)\text{.}
\end{align*}
Now all the summands are non-negative, and it is clear that if $M\ll m\ll M$ and $b\in A\cap[N/2,N]$ then $b-h_d(m)\in[N]$, for some value of the implied constants. This means that at least $w\sigma|A|M$ summands are positive (here we are using Lemma \ref{sievelemma}), and the contribution of each of those is $\gg b_dM^{k-1}$. Therefore,
$$-\int_{\bT}\widehat{g}(\gamma)\overline{\widehat{\mathbf{1}_A}(\gamma)}S(\gamma)d\gamma\gg\frac{1}{wN}\cdot w\sigma |A|M\cdot b_dM^{k-1}-O\left(\frac{\sigma^2}{w}\right)\gg\sigma |A|\text{.}$$
This proves \eqref{part}, under the assumption that $|A\cap[N/2,N]|>|A|/2$. If instead $|A\cap[1,N/2]|>|A|/2$, we derive in a completely analogous way that
$$-\int_{\bT}\widehat{\mathbf{1}_A}(\gamma)\overline{\widehat{g}(\gamma)}S(\gamma)d\gamma\geq \frac{\sigma}{wN}\sum_{\substack{m\in W_d(Y)\cap[1,M]\\a\in A,b\in [N]\\a-b+h_d(m)=0\\h_d'(m)\geq0}}h_d'(m)-O\left(\frac{\sigma^2}{w}\right)\text{.}$$
Now, if $M\ll m\ll M$ and $a\in A\cap[1,N/2]$, for some value of the implied constants, then $a+h_d(m)\in[N]$, and we derive as before that
$$-\int_{\bT}\widehat{\mathbf{1}_A}(\gamma)\overline{\widehat{g}(\gamma)}S(\gamma)d\gamma\gg\sigma |A|\text{,}$$
establishing \eqref{part} in this case as well.

We now claim that there exists a constant $\kappa>0$ such that \eqref{part} remains true if we integrate only over the major arcs $\fM(N,\kappa/\sigma,\kappa/\sigma^{k+1})$, i.e., such that
\begin{equation}
\label{partw}
\int_{\fM(N,\kappa/\sigma,\kappa/\sigma^{k+1})}|\widehat{g}(\gamma)\widehat{\mathbf{1}_A}(\gamma)S(\gamma)|d\gamma\gg\sigma|A|\text{.}
\end{equation}
To prove this, let $r$ be any positive real number such that $2r$ is smaller than the implied constant in \eqref{part}. We apply Lemma \ref{1stc}, yielding $\kappa>0$ such that
$$|S(\gamma)|\leq r\sigma\text{ whenever }\gamma\in\fm\left(N,\frac{\kappa}{\sigma},\frac{\kappa}{\sigma^{k+1}}\right)\text{.}$$
This implies
\begin{align*}
\int_{\fm(N,\kappa/\sigma,\kappa/\sigma^{k+1})}|\widehat{g}(\gamma)\widehat{\mathbf{1}_A}(\gamma)S(\gamma)|d\gamma&\leq r\sigma\int_{\bT}|\widehat{g}(\gamma)\widehat{\mathbf{1}_A}(\gamma)|d\gamma\\
&\leq r\sigma\left(\int_{\bT}|\widehat{\mathbf{1}_A}(\gamma)|^2d\gamma+\sigma\int_{\bT}|\widehat{\mathbf{1}_{[N]}}(\gamma)\widehat{\mathbf{1}_A}(\gamma)|d\gamma\right)\text{.}
\end{align*}
By orthogonality we have
$$\int_\bT|\widehat{\mathbf{1}_A}(\gamma)|^2d\gamma=|A|\text{.}$$
On the other hand, by the Cauchy-Schwarz inequality and orthogonality,
\begin{align*}\int_{\bT}|\widehat{\mathbf{1}_{[N]}}(\gamma)\widehat{\mathbf{1}_A}(\gamma)|d\gamma&\left(\int_{\bT}|\widehat{\mathbf{1}_{[N]}}(\gamma)|^2d\gamma\right)^{1/2}\left(\int_\bT|\widehat{\mathbf{1}_A}(\gamma)|^2d\gamma\right)^{1/2}\\
&=(|A|N)^{1/2}\leq N\text{.}
\end{align*}
We therefore obtain that
$$\int_{\fm(N,\kappa/\sigma,\kappa/\sigma^{k+1})}|\widehat{g}(\gamma)\widehat{\mathbf{1}_A}(\gamma)S(\gamma)|d\gamma\leq 2r\sigma |A|\text{,}$$
which together with \eqref{part}, and recalling the choice of $r$, implies \eqref{partw}.

Now \eqref{partw} can be rewritten as
$$\sum_{q\leq\kappa/\sigma^{k+1}}{\sum_{1\leq a\leq q}}^\ast\int_{\fM_{a,q}(N,\kappa/\sigma)}|\widehat{g}(\gamma)\widehat{\mathbf{1}_A}(\gamma)S(\gamma)|d\gamma\gg\sigma|A|$$
and this implies by the second part of Lemma \ref{2ndc} and Cauchy-Schwarz that
\begin{equation}
\label{ny}
\sum_{q\leq\kappa/\sigma^{k+1}}{\sum_{1\leq a\leq q}}^\ast\frac{1}{q^{1/2}}\left(\underset{\gamma\in\fM_{a,q}(N,\kappa/\sigma)}{\max}\widehat{\mathbf{1}_A(\gamma)}\right)\left(\int_{\fM_{a,q}(N,\kappa/\sigma)}|\widehat{g}(\gamma)|^2d\gamma\right)^{1/2}\gg\frac{\sigma|A|N^{1/2}}{(\log N)^{1/4}}\text{.}
\end{equation}
We observe that those pairs $(a,q)$ for which
$$\int_{\fM_{a,q}(N,\kappa/\sigma)}|\widehat{g}(\gamma)|^2d\gamma\leq\frac{\sigma^{3k+5}N}{\log N}$$
contribute at most
$$\sum_{q\leq\kappa/\sigma^{k+1}}q^{1/2}|A|\frac{\sigma^{(3k+5)/2}N^{1/2}}{(\log N)^{1/2}}\leq\left(\frac{\kappa}{\sigma^{k+1}}\right)^{3/2}|A|\frac{\sigma^{(3k+5)/2}N^{1/2}}{(\log N)^{1/2}}\ll\frac{\sigma|A|N^{1/2}}{(\log N)^{1/2}}\text{.}$$
Here we used the trivial inequality $\widehat{\mathbf{1}_A(\gamma)}\leq|A|$. It follows that the main contribution to \eqref{ny} comes from those pairs $(a,q)$ in the set
$$\Omega=\left\{(a,q):1\leq q\leq\kappa/\sigma^{k+1},1\leq a\leq q,\mathrm{gcd}(a,q)=1, \int_{\fM_{a,q}(N,\kappa/\sigma)}|\widehat{g}(\gamma)|^2d\gamma>\frac{\sigma^{3k+5}N}{\log N}\right\}\text{,}$$
or in other words,
$$\sum_{(a,q)\in\Omega}\frac{1}{q^{1/2}}\left(\underset{\gamma\in\fM_{a,q}(N,\kappa/\sigma)}{\max}\widehat{\mathbf{1}_A(\gamma)}\right)\left(\int_{\fM_{a,q}(N,\kappa/\sigma)}|\widehat{g}(\gamma)|^2d\gamma\right)^{1/2}\gg\frac{\sigma|A|N^{1/2}}{(\log N)^{1/4}}\text{.}$$
We note that $|\widehat{g}(\gamma)|\leq 2N$ for every $\gamma\in\mathbb{T}$, and since the arc $\fM_{a,q}(N,\kappa/\sigma)$ has length $O(1/\sigma)$, we have
$$\frac{\sigma^{(3k+5)/2}N^{1/2}}{(\log N)^{1/2}}<\left(\int_{\fM_{a,q}(N,\kappa/\sigma)}|\widehat{g}(\gamma)|^2d\gamma\right)^{1/2}\ll\frac{N^{1/2}}{\sigma^{1/2}}\quad\text{for every $(a,q)\in\Omega$.}$$
Since $\log N=O(1/\sigma)$, the range determined by the previous inequalities can be covered by $O(\log(1/\sigma))$ intervals of the form $[\sigma N^{1/2}/B,2\sigma N^{1/2}/B]$, for some values of $B$ with $\sigma^{O(1)}\ll B\ll\sigma^{-O(1)}$. Similarly, the interval $[1,\kappa/\sigma^{k+1}]$ can be covered by $O(\log(1/\sigma))$ intervals of the form $[Q,2Q]$, with $Q\ll\sigma^{-O(1)}$. By the Pigeonhole Principle, it follows that there exists a subset $\Gamma\subseteq\Omega$ and values of $B,Q$ with the indicated properties such that, if $(a,q)\in\Gamma$, then $Q\leq q\leq 2Q$ and
$$\sigma N^{1/2}/B\leq \left(\int_{\fM_{a,q}(N,\kappa/\sigma)}|\widehat{g}(\gamma)|^2d\gamma\right)^{1/2}\leq2\sigma N^{1/2}/B\text{,}$$
and moreover
$$\sum_{(a,q)\in\Gamma}\frac{1}{q^{1/2}}\left(\underset{\gamma\in\fM_{a,q}(N,\kappa/\sigma)}{\max}\widehat{\mathbf{1}_A(\gamma)}\right)\left(\int_{\fM_{a,q}(N,\kappa/\sigma)}|\widehat{g}(\gamma)|^2d\gamma\right)^{1/2}\gg\frac{\sigma|A|N^{1/2}}{(\log N)^{1/4}(\log(1/\sigma))^2}\text{.}$$
The previous two inequalities put together imply that
$$\frac{1}{Q^{1/2}}\sum_{(a,q)\in\Gamma}\left(\underset{\gamma\in\fM_{a,q}(N,\kappa/\sigma)}{\max}\widehat{\mathbf{1}_A(\gamma)}\right)\frac{\sigma N^{1/2}}{B}\gg\frac{\sigma|A|N^{1/2}}{(\log N)^{1/4}(\log(1/\sigma))^2}\text{,}$$
which simplifies to the desired result.
\end{proof}

We will now deduce that, if the density increment condition of Lemma \ref{di} does not hold, so that the $L^2$ mass of $\widehat{g}$ around rationals with small denominators is very small, then the set $\Gamma$ constructed in Lemma \ref{rdi} has few pairs $(a,q)$ with a given value of $q$.

\begin{lemma}
\label{cor0}
Keep the assumptions of Lemma \ref{rdi}, and let $\nu\geq N^{-1/2}$. Then one of the following holds:
\begin{enumerate}[label=(\roman*)]
\item There exists $N^\ast\gg\nu\sigma^{O(1)}N$, an integer $q\ll\sigma^{-O(1)}$ and a $h_{qd}$-free subset $A^\ast\subseteq[N^\ast]$ with density
$$\sigma^\ast\geq\left(1+\frac{\nu}{73}\right)\sigma\text{;}$$
\item There exist $B$, $Q$ and a set $\Gamma$ as in the conclusion of Lemma \ref{rdi} with the additional property that, for every $q\leq Q$,
$$|\{a:(a,q)\in\Gamma\}|\ll\nu B^2\text{.}$$
\end{enumerate}
\end{lemma}

\begin{proof}
We assume that condition (i) does not hold; note that, given that $\lambda(q)\leq q^d$, the condition $\nu\sigma N\gg K\lambda(q)$ holds automatically whenever $q,K\ll\sigma^{-O(1)}$, so that Lemma \ref{di} implies that, for such $q$ and $K$,
$$\sum_{\substack{1\leq a\leq q\\(a,q)=1}}\int_{\fM_{a,q}(N,K)}\left|\widehat{g}(\gamma)\right|^2d\gamma<\nu\sigma|A|\text{.}$$
We now consider values $B$ and $Q$ and a set $\Gamma$ as in the conclusion of Lemma \ref{rdi}, and for a value of $q\leq Q$ we add up the inequality
$$\int_{\fM_{a,q}(N,\kappa/\sigma)}|\widehat{g}(\gamma)|^2d\gamma\gg\frac{\sigma|A|}{B^2}$$
over all values of $a$ for which $(a,q)\in\Gamma$. This yields
$$|\{a:(a,q)\in\Gamma\}|\frac{\sigma|A|}{B^2}\ll\nu\sigma|A|\text{,}$$
which simplifies to (ii).
\end{proof}

\section{Additive Energy bounds and end of the proof}
We quote the following result from \cite{BM}, which essentially says that a set of frequencies at which the Fourier coefficients of a set $A$ are large has high additive energy.
\begin{lemma}
\label{ch}
Let $A\subseteq[N]$ be a set of density $\sigma=|A|/N$. Let set $S\subseteq\mathbb{T}$ be a finite set. Then, for every $m\geq 1$,
$$\sum_{\gamma\in S}|\widehat{\mathbf{1}_A}(\gamma)|\ll|A|\sigma^{-1/2m}E_{2m}(S;1/2N)^{1/2m}\text{,}$$
where we define
$$E_{2m}(S;\delta)=|\{(b_1,\ldots,b_{2m})\in S^{2m}:\|b_1+\cdots+b_m-b_{m+1}-\cdots-b_{2m}\|\leq\delta\}|\text{.}$$
\end{lemma}
\begin{proof}
This is (\cite{BM}, Lemma 7), but we provide a proof for completeness. For each $\gamma$ let $\omega_{\gamma}$ be a complex number with absolute value $1$ for which $\omega_\gamma\widehat{\mathbf{1}_A(\gamma)}=|\widehat{\mathbf{1}_A(\gamma)}|$. By Hölder's inequality,
\begin{align*}\sum_{\gamma\in S}|\widehat{\mathbf{1}_A}(\gamma)|&=\sum_{\gamma in S}\omega_\gamma\sum_{a\in A}e(a\gamma)\\
&\leq |A|^{1-1/2m}\left(\sum_{a\in A}\left|\sum_{\gamma\in S}\omega_\gamma e(a\gamma)\right|^{2m}\right)^{1/2m}\text{.}
\end{align*}
We now consider the function $\psi$ on $\mathbb{R}$ defined by $\psi(t)=\sin(\pi t)^2/(\pi t)^2$, whose Fourier transform $\widehat{\psi}$ satisfies
$$\widehat{\psi}(\xi)=\begin{cases}1-|\xi|&\text{ if }|\xi|\leq 1\\
0&\text{ otherwise. }\\
\end{cases}
$$
Since $\psi$ is non-negative and $\psi(t)\gg 1$ for $t\in[-1/2,1/2]$, it follows that
\begin{align*}
\sum_{a\in A}\left|\sum_{\gamma\in S}\omega_\gamma e(a\gamma)\right|^{2m}&\ll\sum_{n\in\mathbb{Z}}\psi\left(\frac{n}{2N}\right)\left|\sum_{\gamma\in S}\omega_\gamma e(a\gamma)\right|^{2m}\\
&\ll\sum_{(b_1,\ldots,b_{2m})\in S^{2m}}\left|\sum_{n\in\mathbb{Z}}\psi\left(\frac{n}{2N}\right)e\left(n(b_1+\cdots+b_m-b_{m+1}-\cdots-b_{2m})\right)\right|\text{.}
\end{align*}
We now apply the Poisson summation formula (\cite{IK}, Theorem 4.4) to the inner sum, which, together with the fact that $\widehat{\psi}$ has support $[-1,1]$, implies that the above is
\begin{align*}&\ll\sum_{(b_1,\ldots,b_{2m})\in S^{2m}}N\left|\sum_{h\in\mathbb{Z}}\widehat{\psi}\left(2N(b_1+\cdots+b_m-b_{m+1}-\cdots-b_{2m})\right)\right|\\
&\ll N|\{(b_1,\ldots,b_{2m})\in S^{2m}:\|b_1+\cdots+b_m-b_{m+1}-\cdots-b_{2m}\|\leq1/2N\}|\text{.}
\end{align*}
Substituting this into our application of Hölder finishes the proof.
\end{proof}
Lemma \ref{ch} runs into some tension against the following additive energy bound of Bloom and Maynard if the large Fourier coefficients occur at rationals with different denominators, which we know will be the case in the absence of a good density increment, according to Lemma \ref{cor0}.
\begin{lemma}
\label{newbm}
Let $Q\geq 4$ and $m\geq2$. Let $S$ be a set of positive rational frequencies in $\mathbb{T}$ with denominator at most $Q$ contained in an interval of length $1/8m$ and assume that $S$ contains at most $n$ frequencies with denominator $q$ for every $q\leq Q$. Then we have
$$E_{2m}(S;0)\leq(\log Q)^{C^m}(Qn)^m$$
for some absolute constant $C$.
\end{lemma}
\begin{proof}
Let $S'$ be a set of representatives on the real line of the elements of $S$, contained in an interval of length $1/8m$. The number $E_{2m}(S;0)$ counts $2m$-tuples $(b_1,\ldots,b_{2m})\in S'^{2m}$ for which
$$b_1+\cdots+b_m-b_{m+1}-\cdots-b_{2m}\in\mathbb{Z}\text{.}$$
But the above expression lies in $[-1/4,1/4]$, so if it is an integer then it must equal $0$. Therefore the result follows from Theorem 2 of \cite{BM}.
\end{proof}

We are now ready to state and prove the next result which essentially determines the quality of the density increment that the previous additive energy bounds allow us to obtain. It is similar in spirit to Lemma 8 of \cite{BM}, with the differences arising from our slightly worse bound for the sum in Lemma \ref{rdi}(i) (Bloom and Maynard manage to avoid the power of $\log N$ in the denominator). However this slight worsening of the quality of the density increment, as measured by the parameter $\nu$, turns out not to be detrimental to the quality of the final bound.

\begin{lemma}
\label{finpiece}
Let $N\gg_h 1$, and suppose as usual that $d\leq N^\rho$. Then, for some absolute constant $c>0$, if we set
$$\nu=\frac{1}{(\log N)^{1/2}}\exp\left(-c\frac{\log(1/\sigma)}{\log\log(1/\sigma)}\right)\text{,}$$
then there exist $N^\ast\gg\sigma^{O(1)}N/(\log N)^{1/2}$, an integer $q\ll\sigma^{-O(1)}$ and a $h_{qd}$-free subset $A^\ast\subseteq[N^\ast]$ with density
$$\sigma^\ast\geq\left(1+\frac{\nu}{73}\right)\sigma\text{.}$$
\end{lemma}
\begin{proof}
Let $m\geq2$ be an integer to be chosen later. We apply Lemma \ref{cor0}; our goal is to prove that case (i) must hold for the presented choice of $\nu$. Assume it does not; observe that, by an application of the Pigeonhole Principle, there exists a subset $\Gamma'\subseteq\Gamma$ such that the frequencies $\gamma_{a,q}$ for $(a,q)\in\Gamma'$ lie on an interval of length $1/8m$ and such that
$$\sum_{(a,q)\in\Gamma}|\widehat{\mathbf{1}_A}(\gamma_{a,q})|\gg\frac{B|A|Q^{1/2}}{m(\log N)^{1/4}(\log(1/\sigma))^2}\text{.}$$
Let $S=\{\gamma_{a,q}:(a,q)\in\Gamma'\}$ and let $T=\{a/q:(a,q)\in\Gamma\}$. It follows from Lemma \ref{ch} that
$$\frac{B|A|Q^{1/2}}{m(\log N)^{1/4}(\log(1/\sigma))^2}\ll|A|\sigma^{-1/2m}E_{2m}(S;1/2N)^{1/2m}\text{.}$$
Suppose $(a_1,q_1),\ldots,(a_{2m},q_{2m})$ are such that
\begin{equation}
\label{adencon}
\|\gamma_{a_1,q_1}+\cdots+\gamma_{a_m,q_m}-\gamma_{a_{m+1},q_{m+1}}-\cdots-\gamma_{a_{2m},q_{2m}}\|\leq\frac{1}{2N}\text{.}
\end{equation}
Since $\gamma_{a,q}$ differs from $a/q$ by at most $\kappa/\sigma N$, it follows that
$$\left\|\frac{a_1}{q_1}+\cdots+\frac{a_m}{q_m}-\frac{a_{m+1}}{q_{m+1}}-\cdots-\frac{a_{2m}}{q_{2m}}\right\|\leq\frac{1}{2N}+\frac{\kappa m}{\sigma N}\ll\frac{m}{\sigma N}\text{.}$$
On the other hand, the sum above, if nonzero, has denominator at most $Q^{2m}$. Since $Q\ll\sigma^{-O(1)}$, it would follow that $\sigma N\ll m\sigma^{-O(m)}$; but, as we will see, our final choice of $m$ will be such that $m\ll\log\log(1/\sigma)$, so that these would contradict our standing assumptions on the relative sizes of $N$ and $\sigma$. It follows that, whenever \eqref{adencon} holds, we have
$$\frac{a_1}{q_1}+\cdots+\frac{a_m}{q_m}-\frac{a_{m+1}}{q_{m+1}}-\cdots-\frac{a_{2m}}{q_{2m}}=0\text{.}$$
Therefore $E_{2m}(S;1/2N)\leq E_{2m}(T;0)$, and it follows that
$$\frac{B|A|Q^{1/2}}{m(\log N)^{1/4}(\log(1/\sigma))^2}\ll|A|\sigma^{-1/2m}E_{2m}(T;0)^{1/2m}\text{.}$$
We now apply Lemma \ref{newbm} to bound the additive energy, noting that according to the conclusion of Lemma \ref{cor0} we can take $n\ll\nu B^2$. It follows that
$$\frac{B|A|Q^{1/2}}{m(\log N)^{1/4}(\log(1/\sigma))^2}\ll|A|\sigma^{-1/2m}E_{2m}(T;0)^{1/2m}1\leq |A|\sigma^{-1/2m}(\log Q)^{C^m/2m}\nu^{1/2}Q^{1/2}B\text{,}$$
and now - and this is the key reason why we needed to obtain square root cancellation in our exponential sums! - we can cancel the powers of $Q$ on both sides (as well as some other factors) and we obtain
$$\nu\gg\frac{\sigma^{1/m}}{(\log N)^{1/2}m^2\log(1/\sigma)^4(\log Q)^{C^m/m}}\text{.}$$
Recalling that $Q\ll\sigma^{-O(1)}$, we see that, if $c'>0$ is sufficiently small, then setting $m=\lceil c'\log\log(1/\sigma)\rceil$, this yields
$$\nu\gg\frac{1}{(\log N)^{1/2}}\exp\left(-O\left(\frac{\log(1/\sigma)}{\log\log(1/\sigma)}\right)\right)\text{.}$$
This means that case (ii) in Lemma \ref{cor0} is not compatible with the given choice of $\nu$ if $c$ is sufficiently small, so that case (i) must hold, which is precisely what we want.
\end{proof}

We now apply an iterative density increment argument to deduce our main result, Theorem \ref{main}.

\begin{proof}
Assume that $A\subseteq[N]$ is $h$-free with density $\sigma=|A|/N$. We may as well assume that $N\gg_h1$. We set
$$\nu=\frac{1}{(\log N)^{1/2}}\exp\left(-c\frac{\log(1/\sigma)}{\log\log(1/\sigma)}\right)\text{,}$$
as given by Lemma \ref{finpiece}. Applying Lemma \ref{finpiece} iteratively $t$ times produces:
\begin{itemize}
\item A sequence of positive integers $N_0,N_1,\ldots,N_t$ with $N_0=N$ and $N_{i+1}\geq\sigma^{O(1)}N_i/(\log N)^{1/2}$;
\item A sequence of positive integers $d_0,\ldots,d_t$ with $d_0=1$ and $d_{i+1}\leq\sigma^{-O(1)}d_i$;
\item A sequence of subsets $A_i\subseteq[N_i]$ such that $A_i$ is $h_{d_i}$-free and has density $\sigma_i$, with
$$\sigma_{i+1}\geq\left(1+\frac{\nu}{73}\right)\sigma_i\text{.}$$
\end{itemize}
We perform this iteration until we find $t$ such that $N_t<N^{1/2}$. Until then the conditions of Lemma \ref{finpiece} remain in place; for expository purposes we will defer the verification that the technical condition $d_t\leq N^{\rho}$ remains true when that holds to the end of the proof.

We now observe that
$$1\geq\sigma_t\geq\left(1+\frac{\nu}{73}\right)^t\sigma\text{,}$$
whence
\begin{equation}
\label{tb}
t\leq\frac{\log(1/\sigma)}{\log(1+\nu/73)}\ll\nu^{-1}\log(1/\sigma)\text{.}
\end{equation}
On the other hand, one has
$$N^{1/2}>N_t\geq\frac{\sigma^{O(t)}}{(\log N)^{t/2}}N\text{,}$$
which implies, after canceling $N^{1/2}$ and taking logarithms,
$$\log N\ll t\log(1/\sigma)+t\log\log N\text{.}$$
Now using the bound for $t$ given by \eqref{tb}, it follows that
$$\log N\ll\nu^{-1}\log(1/\sigma)^2+\nu^{-1}\log(1/\sigma)\log\log N\text{,}$$
whence
\begin{equation}
\label{spl}
\log N\ll\nu^{-1}\log(1/\sigma)^2\text{ or }\log N\ll\nu^{-1}\log(1/\sigma)\log\log N\text{.}
\end{equation}

If the former holds, then unraveling the definition of $\nu$ yields
$$(\log N)^{1/2}\ll\exp\left(c\frac{\log(1/\sigma)}{\log\log(1/\sigma)}\right)\log(1/\sigma)^2$$
and taking logarithms from both sides we obtain
$$\log\log N\ll\frac{\log(1/\sigma)}{\log\log(1/\sigma)}+\log\log(1/\sigma)\ll\frac{\log(1/\sigma)}{\log\log(1/\sigma)}\text{.}$$
The above is easily seen to imply that $\log(1/\sigma)\gg\log\log N\cdot\log\log\log N$, which is a restatement of the desired result.

If, on the other hand, the latter holds in \eqref{spl}, then similarly unraveling the definition of $\nu$ yields
$$(\log N)^{1/2}\ll\exp\left(c\frac{\log(1/\sigma)}{\log\log(1/\sigma)}\right)\log(1/\sigma)\log\log N\text{,}$$
and taking logarithms gives
$$\log\log N\ll\frac{\log(1/\sigma)}{\log\log(1/\sigma)}+\log\log(1/\sigma)+\log\log\log N\text{,}$$
and since $\log\log\log N=o(\log\log N)$, it follows that
$$\log\log N\ll\frac{\log(1/\sigma)}{\log\log(1/\sigma)}\text{,}$$
which, as we have seen, implies the result.

It remains only to check that $d_t\leq N^{\rho}$, which it suffices to do under the assumption that $\sigma\geq (\log N)^{-\log\log\log N/4c}$, where $c$ is the constant appearing in the statement of Lemma \ref{finpiece}. Under that assumption, it is easily seen that
$$\frac{\log(1/\sigma)}{\log\log(1/\sigma)}\leq\frac{1}{4c}\log\log N$$
(if $N$ is large enough). Now observe that $d_t\leq\sigma^{-O(t)}$, and
\begin{align*}t&\leq(\log N)^{1/2}\exp\left(c\frac{\log(1/\sigma)}{\log\log(1/\sigma)}\right)\log(1/\sigma)\\
&\leq(\log N)^{1/2}\exp\left(\frac{1}{4}\log\log N\right)\log(1/\sigma)\\
&=(\log N)^{3/4}\log(1/\sigma)\\
&\ll(\log N)^{3/4}\log\log N\log\log\log N\text{.}
\end{align*}
This implies that
$$d_t\leq(\log N)^{O((\log N)^{3/4}\log\log N(\log\log\log N)^2)}=\exp(O((\log N)^{3/4}(\log\log N)^2(\log\log\log N)^2))\text{,}$$
and the above is smaller than $N^{\rho}$ if $N$ is sufficiently large.
\end{proof}

\section{Acknowledgements}
The author was funded through the Engineering
and Physical Sciences Research Council Doctoral Training Partnership at the
University of Warwick. Special thanks are due to Sam Chow for bringing this problem to the attention of the author.

\providecommand{\bysame}{\leavevmode\hbox to3em{\hrulefill}\thinspace}

\end{document}